\documentclass[a4paper, 12pt]{article}
\usepackage{amssymb, amsmath, amsthm}
\usepackage[english]{babel}
\usepackage{enumerate}

\large\normalsize \hbadness3000 \vbadness30000
\def\mytitle#1{\setcounter{equation}{0}
\setcounter{footnote}{0}
\begin{center}\Large\textbf{#1}\end{center}
\vspace{0.25cm}}
\def\myname#1{\centerline{{\large #1}}\vspace{-0.13cm}}

\newtheorem{theorem}{Theorem}[section]
\newtheorem{corollary}[theorem]{Corollary}
\newtheorem{lemma}[theorem]{Lemma}
\newtheorem{proposition}[theorem]{Proposition}

\theoremstyle{definition}
\newtheorem{definition}[theorem]{Definition}

\newcommand\N{\mathfrak{N}}
\newcommand\F{\mathbf{F}}
\newcommand\T{\mathbf{T}}
\renewenvironment{proof}[1][Proof]{\noindent\textbf{#1: }}{\hspace{\stretch{1}}\rule{0.5em}{0.5em}}
\newcommand\bibitemart{\bibitem}
\newcommand\bibitemproc{\bibitem}
\newcommand\bibitembook{\bibitem}

\begin{document}
\mytitle{The Dynamics of the Forest Graph Operator}
\myname{
Suresh Dara,$^a$ S.\ M.\ Hegde,$^a$ Venkateshwarlu Deva,$^b$ }
\vspace{0.25cm}
\myname{S.\ B.\ Rao,$^b$ Thomas Zaslavsky$^c$}

\begin{center}
$^a$Department of Mathematical and Computational Sciences, \\National Institute of Technology Karnataka, \\Surathkal, Mangalore-575025, India
\\[0.25cm]
$^b$C R Rao Advanced Institute of Mathematics, Statistics and Computer Science,\\
Hyderabad-500 046, India\\[0.25cm]
$^c$Binghamton University, Binghamton, NY, U.S.A. 13902-6000\\[0.5cm]
{\em E-mail:} suresh.dara@gmail.com,  smhegde@nitk.ac.in,\\ venky477@gmail.com, siddanib@yahoo.co.in, zaslav@math.binghamton.edu
\end{center}

\small
\begin{abstract}
In 1966, Cummins introduced the ``tree graph'': the tree graph $\T(G)$ of a graph $G$ (possibly infinite) has all its spanning trees as vertices, and distinct such trees correspond to adjacent vertices if they differ in just one edge, i.e., two spanning trees $T_1$ and $T_2$ are adjacent if $T_2 = T_1 -e +f$ for some edges $e\in T_1$ and $f\notin T_1$. The tree graph of a connected graph need not be connected. To obviate this difficulty we define the ``forest graph'': let $G$ be a labeled graph of order $\alpha$, finite or infinite, and let $\N(G)$ be the set of all labeled maximal forests of $G$. The forest graph of $G$, denoted by $\F(G)$, is the graph with vertex set $\N(G)$ in which two maximal forests $F_1$, $F_2$ of $G$ form an edge if and only if they differ exactly by one edge, i.e., $F_2 = F_1 -e +f$ for some edges $e\in F_1$ and $f\notin F_1$.

Using the theory of cardinal numbers, Zorn's lemma, transfinite induction, the axiom of choice and the well-ordering principle, we determine the $\F$-convergence, $\F$-divergence, $\F$-depth and $\F$-stability of any graph $G$. In particular it is shown that a graph $G$ (finite or infinite) is $\F$-convergent if and only if $G$ has at most one cycle of length 3. The $\F$-stable graphs are precisely $K_3$ and $K_1$. The $\F$-depth of any graph $G$ different from $K_3$ and $K_1$ is finite.  We also determine various parameters of $\F(G)$ for an infinite graph $G$, including the number, order, size, and degree of its components.

\end{abstract}

\begin{quote}

\textbf{Mathematics Subject Clasification 2010:} {Primary 05C76; Secondary 05C05, 05C63}
\end{quote}

\normalsize


\section{Introduction}

A \emph{graph dynamical system} is a set $X$ of graphs together with a mapping $\phi: X\rightarrow X$ (see Prisner \cite{prisner1995graph}).  We investigate the graph dynamical system on finite and infinite graphs defined by the forest graph operator $\F$, which transforms $G$ to its graph of maximal forests.  

Let $G$ be a labeled graph of order $\alpha$, finite or infinite. (All our graphs are labeled.) A \emph{spanning tree} of $G$ is a connected, acyclic, spanning subgraph of $G$; it exists if and only if $G$ is connected. 
Any acyclic subgraph of $G$, connected or not, is called a \emph{forest} of $G$. A forest $F$ of $G$ is said to be \emph{maximal} if there is no forest $F'$ of $G$ such that $F$ is a proper subgraph of $F'$. 
The tree graph $\T(G)$ of $G$ has all the spanning trees of $G$ as vertices, and distinct such trees are adjacent vertices if they differ in just one edge \cite{prisner1995graph, sureshdara2010icm}; i.e., two spanning trees $T_1$ and $T_2$ are adjacent if $T_2 = T_1 -e +f$ for some edges $e\in T_1$ and $f\notin T_1$. The \emph{iterated tree graphs} of $G$ are defined by $\T^0(G) =G$ and $\T^n(G) = \T(\T^{n-1}(G))$ for $n>0$. 
There are several results on tree graphs. See \cite{broersma1996connectivity, zhang1986connectivity, liu1988connectivities} for connectivity of the tree graph, \cite{grimmett1976upper, rodpet1997sharp, teranishi2005number, zhang2005bound, das2007sharp, feng2008sharp, feng2010number, das2013number, fengTAmatching} for bounds on the order of $\T(G)$ (that is, on the number of spanning trees of $G$), \cite{cummins1966hamilton, shank1968note} for Hamilton circuits in a tree graph.

There is one difficulty with iterating the tree graph operator.  The tree graph of an infinite connected graph need not be connected \cite{cummins1966hamilton, shank1968note}, so $\T^2(G)$ may be undefined. For example, $\T(K_{\aleph_0})$ is disconnected (see Corollary \ref{l2.d} in this paper; $\aleph_0$ denotes the cardinality of the set $\mathbb N$ of natural numbers); therefore $\T^2(K_{\aleph_0})$ is not defined. 
To obviate this difficulty with iterated tree graphs, and inspired by the tree graph operator $\T$, we define a forest graph operator. 
Let $\N(G)$ be the set of all maximal forests of $G$. The \emph{forest graph} of $G$, denoted by $\F(G)$, is the graph with vertex set $\N(G)$ in which two maximal forests $F_1$, $F_2$ form an edge if and only if they differ by exactly one edge. The \emph{forest graph operator} (or \emph{maximal forest operator}) on graphs, $G \mapsto \F(G)$, is denoted by $\F$. 
Zorn's lemma implies that every connected graph contains a spanning tree (see \cite{diestel2005graph}); similarly, every graph has a maximal forest.  Hence, the forest graph always exists.  Since when $G$ is connected, maximal forests are the same as spanning trees, then $\F(G)=\T(G)$; that is, the tree graph is a special case of the forest graph. 
We write $\F^2(G)$ to denote $\F(\F(G))$, and in general $\F^n(G) = \F(\F^{n-1}(G))$ for $n \geq 1$, with $\F^0(G) = G$.

\begin{definition}\label{d2}
A graph $G$ is said to be \emph{$\F$-convergent} if $\{\F^n(G): n \in \mathbb{N}\}$ is finite; otherwise it is \emph{$\F$-divergent}.

A graph $H$ is said to be an \emph{$\F$-root} of $G$ if $\F(H)$ is isomorphic to $G$, $\F(H)\cong G$.
The \emph{$\F$-depth} of $G$ is 
$$\sup \{n \in  \mathbb{N} : G\cong \F^n(H) \text{ for some graph } H\}.$$
The $\F$-depth of a graph $G$ that has no $\F$-root is said to be zero.

The graph $G$ is said to be \emph{$\F$-periodic} if there exists a positive integer $n$ such that $\F^n(G)=G$.  The least such integer is called the \emph{$\F$-periodicity} of $G$.  
If $n=1$, $G$ is called \emph{$\F$-stable}.
\end{definition}

This paper is organized as follows. In Section \ref{prelim} we give some basic results. 
In later sections, using Zorn's lemma, transfinite induction, the well ordering principle and the theory of cardinal numbers, we study the number of $\F$-roots and determine the $\F$-convergence, $\F$-divergence, $\F$-depth and $\F$-stability of any graph $G$. In particular we show that: 
i) A graph $G$ is $\F$-convergent if and only if $G$ has at most one cycle of length 3. 
ii) The $\F$-depth of any graph $G$ different from $K_3$ and $K_1$ is finite. 
iii) The $\F$-stable graphs are precisely $K_3$ and $K_1$. 
iv) A graph that has one $\F$-root has innumerably many, but only some $\F$-roots are important.



\section{Preliminaries}\label{prelim}

For standard notation and terminology in graph theory we follow Diestel \cite{diestel2005graph} and Prisner \cite{prisner1995graph}.
 
Some elementary properties of infinite cardinal numbers that we use are (see, e.g., Kamke \cite{kamke1950theory}):
 \begin{enumerate}[\quad(1)]
 \item $\alpha+\beta = \alpha.\beta = \max(\alpha,\beta)$ if $\alpha, \beta$ are cardinal numbers and $\beta$ is infinite.  In particular, $2.\beta=\aleph_0.\beta=\beta$.
 \item $\beta^n = \beta$ if $\beta$ is an infinite cardinal and $n$ is a positive integer.
 \item $\beta<2^\beta$ for every cardinal number.
 \item The number of finite subsets of an infinite set of cardinality $\beta$ is equal to $\beta$.
 \end{enumerate}

We consider finite and infinite labeled graphs \emph{without multiple edges or loops}.  An \emph{isthmus} of a graph $G$ is an edge $e$ such that deleting $e$ divides one component of $G$ into two of $G-e$.  Equivalently, an isthmus is an edge that belongs to no cycle.  Each isthmus is in every maximal forest, but no non-isthmus is.

Let  $\mathfrak{C}(G)$ and $\N(G)$ denote the set of all possible cycles and the set of all maximal forests of a graph $G$, respectively.  
Note that a maximal forest of $G$ consists of a spanning tree in each component of $G$.  A fundamental fact, whose proof is similar to that of the existence of a maximal forest, is the following forest extension lemma:

\begin{lemma}\label{forest-extension}
In any graph $G$, every forest is contained in a maximal forest.
\end{lemma}



\begin{lemma}\label{l2.5}
If $G$ is a complete graph of infinite order $\alpha$, then $|\N(G)| = 2^{\alpha}$.
\end{lemma}

\begin{proof}
Let $G=(V,E)$ be a complete graph of order $\alpha$ ($\alpha$ infinite), i.e., $G=K_{\alpha}$. 
Let $v_1$, $v_2$ be two vertices of $G$ and $V'=V\setminus\{v_1, v_2\}$. Then for every $A\subseteq V'$ there is a spanning tree $T_A$ such that every vertex of $A$ is adjacent only to $v_1$ and every vertex of $V'\setminus A$ is adjacent only to $v_2$. 
It is easy to see that $T_A \neq T_B$ whenever $A\neq B$. 
As the cardinality of the power set of $V'$ is $2^\alpha$, there are at least $2^\alpha$ spanning trees of $G$.
Since $G$ is connected, the maximal forests are the spanning trees; therefore $|\N(G)| \geq 2^{\alpha} $.
Since the degree of each vertex is $\alpha$ and $G$ contains $\alpha$ vertices, the total number of edges in $G$ is $\alpha.\alpha = \alpha$. 
The edge set of a maximal forest of $G$ is a subset of $E$ and the number of all
possible subsets of $E$ is $2^{\alpha}$. Therefore, $G$ has at most
$2^{\alpha}$ maximal forests, i.e., $|\N(G)| \leq 2^{\alpha}$. Hence $|\N(G)| = 2^{\alpha}$.
\end{proof}


For two maximal forests of $G$, $F_1$ and $F_2$, let $d(F_1,F_2)$ denote the distance between them in $\F(G)$.  We connect this distance to the number of edges by which $F_1,F_2$ differ; the result is elementary but we could not find it anywhere in the literature.  We say $F_1, F_2$ \emph{differ by $l$ edges} if $|E(F_1) \setminus E(F_2)|=|E(F_2) \setminus E(F_1)|=l$.

\begin{lemma}\label{l2.3}
Let $l$ be a natural number.  For two maximal forests $F_1, F_2$ of a graph $G$, if $|E(F_1) \setminus E(F_2)|=l$, then $|E(F_2) \setminus E(F_1)|=l$.  
Furthermore, $F_1$ and $F_2$ differ by exactly $l$ edges if and only if $d(F_1,F_2)=l$.
\end{lemma}

We cannot apply to an infinite graph the simple proof for finite graphs, in which the number of edges in a maximal forest is given by a formula.  Therefore, we prove the lemma by edge exchange.

\begin{proof}
We prove the first part by induction on $l$.  Let $F_1, F_2$ be maximal forests of $G$ and let $E(F_1) \setminus E(F_2)=\{e_{1}',e_{2}',\ldots,e_{k}'\}$, $E(F_2) \setminus E(F_1)=\{e_{1},e_{2},\ldots,e_{l}\}$.  If $l=0$ then $k=0=l$ because $F_2 = F_1$.  Suppose $l>0$; then $k>0$ also.  Deleting $e_l$ from $F_2$ divides a tree of $F_2$ into two trees.  Since these trees are in the same component of $G$, there is an edge of $F_1$ that connects them; this edge is not $e_1$ so it is not in $F_2$; therefore, it is an $e_i'$, say $e_k'$.  Let $F_2' = F_2 - e_l + e_k'$.  Then $E(F_1) \setminus E(F_2')=\{e_{1}',e_{2}',\ldots,e_{k-1}'\}$, $E(F_2) \setminus E(F_1)=\{e_{1},e_{2},\ldots,e_{l-1}\}$.  By induction, $k-1=l-1$.

We also prove the second part by induction on $l$.  Assume $F_1,F_2$ differ by exactly $l$ edges and define $F_2'$ as above.  If $l=0,1$, clearly $d(F_1,F_2)=l$.  Suppose $l>1$.  
In a shortest path from $F_1$ to $F_2$, whose length is $d(F_1,F_2)$, each successive edge of the path can increase the number of edges not in $F_1$ by at most 1.  Therefore, $F_1$ and $F_2$ differ by at most $d(F_1,F_2)$ edges.  That is, $l\leq d(F_1,F_2)$.
Conversely, $d(F_1,F_2')=l-1$ by induction and there is a path in $\F(G)$ from $F_1$ to $F_2'$ of length $l-1$, then continuing to $F_2$ and having total length $l$.  Thus, $d(F_1,F_2)\leq l$.
\end{proof}


From the above lemma we have two corollaries.

\begin{corollary}\label{l2.4}
For any graph $G$, $\F(G)$ is connected if and only if any two maximal forests of $G$ differ by at most a finite number of edges. 
\end{corollary}

\begin{corollary}\label{l2.d}
If $G=K_\alpha$, $\alpha$ infinite, then $\F(G)$ is disconnected.
\end{corollary}


\begin{lemma}\label{edge-count}
Let $G$ be a graph with $\alpha$ vertices and $\beta$ edges and with no isolated vertices.  If either $\alpha$ or $\beta$ is infinite, then $\alpha=\beta$.
\end{lemma}

\begin{proof}
We know that $|E(G)|\leq|V(G)|^2$, i.e., $\beta\leq\alpha^2$ so if $\beta$ is infinite, $\alpha$ must also be infinite.  We also know, since each edge has two endpoints, that $|V(G)|\leq2|E(G)|$, i.e., $\alpha\leq2.\beta$ so if $\alpha$ is infinite, then $\beta$ must be infinite.  
Now assuming both are infinite, $\alpha^2=\alpha$ and $2.\beta=\beta$, hence $\alpha=\beta$.
\end{proof}


The following lemmas are needed in connection with $\F$-convergence and $\F$-divergence in Section \ref{convergence} and $\F$-depth in Section \ref{depth}.

\begin{lemma}\label{l2.1}
Let $G$ be a graph. If $K_n$ (for finite $n\geq2$) is a subgraph of $G$, then $K_{\lfloor{n^2}/{4}\rfloor}$ is a subgraph of $\F(G)$.
\end{lemma}

\begin{proof}
Let $G$ be a graph such that $K_n$ ($n\geq2$, finite) is a subgraph of $G$ with vertex labels $v_1, v_2, \ldots, v_n$. Then there is a path $L = v_1, v_2, \ldots ,v_n$ of order $n$ in $G$. Let $F$ be a maximal forest of $G$ such that $F$ contains the path $L$. In $F$ if we replace the edge $v_{\lfloor{n}/{2}\rfloor}v_{\lfloor{n}/{2}\rfloor+1}$ by any other edge $v_iv_j$ where $i=1,\ldots,\lfloor{n}/{2}\rfloor$ and $j=\lfloor{n}/{2}\rfloor+1,\ldots, n$, we get a maximal forest $F_{ij}$.  Since there are $\lfloor {n^2}/{4} \rfloor$ such edges $v_iv_j$, there are $\lfloor {n^2}/{4} \rfloor$ maximal forests $F_{ij}$ (of which one is $F$).  Any two forests $F_{ij}$ differ by one edge.  It follows that they form a complete subgraph in $\F(G)$. Therefore $K_{\lfloor{n^2}/{4}\rfloor}$ is a subgraph of $\F(G)$.
\end{proof}


\begin{lemma}\label{cycle-kn}
If $G$ has a cycle of (finite) length $n$ with $n\geq3$, then $\F(G)$ contains $K_n$.
\end{lemma}

\begin{proof}
Suppose that $G$ has a cycle $C_n$ of length $n$ with edge set $\{e_1,e_2,\ldots,e_n\}$. 
Let $P_i=C_n - e_i$ for $i=1,2, \dots, n$ and let $F_1$ be a maximal forest of $G$ containing the path $P_1$.  Define $F_i=F_1 \setminus P_1 \cup P_i$ for $i=2, 3, \dots, n$.  
These $F_i$'s are maximal forests of $G$ and any two of them differ by exactly one edge, so they form a complete graph $K_n$ in $\F(G)$.  
\end{proof}

In particular, $\F(C_n)=K_n$.

\begin{lemma}\label{k-cayley}
Suppose that $G$ contains $K_n$, where $n\geq3$.  Then $\F^2(G)$ contains $K_{n^{n-2}}$.
\end{lemma}

\begin{proof}
Cayley's formula states that $K_n$ has $n^{n-2}$ spanning trees.  Cummins \cite{cummins1966hamilton} proved that the tree graph of a finite connected graph is Hamiltonian.  Therefore, $\F(K_n)$ contains $C_{n^{n-2}}$.  
Let $F_{T_0}$ be a spanning tree of $G$ that extends one of the spanning trees $T_0$ of the $K_n$ subgraph.  Replacing the edges of $T_0$ in $F_{T_0}$ by the edges of any other spanning tree $T$ of $K_n$, we have a spanning tree $F_T$ that contains $T$.  
The $F_T$'s for all spanning trees $T$ of $K_n$ are $n^{n-2}$ spanning trees of $G$ that differ only within $K_n$; thus, the graph of the $F_T$'s is the same as the graph of the $T$'s, which is Hamiltonian.  That is, $\F(G)$ contains $C_{n^{n-2}}$.  By Lemma \ref{cycle-kn}, $\F^2(G)$ contains $K_{n^{n-2}}$.
\end{proof}

We do not know exactly what graphs $\F(K_n)$ and $\F^2(K_n)$ are.

\begin{lemma}\label{triangles-k9}
If $G$ has two edge disjoint triangles, then $\F^2(G)$ contains $K_9$.
\end{lemma}

\begin{proof}
Suppose that $G$ has two edge disjoint triangles whose edges are $e_1,e_2,e_3$ and $f_1,f_2,f_3$, respectively.  The union of the triangles has exactly 9 maximal forests $F_{ij}'$, obtained by deleting one $e_i$ and one $f_j$ from the triangles.  Extend $F_{11}'$ to a maximal forest $F_{11}$ and let $F_{ij}$ be the maximal forest $F_{11}\setminus E(F_{11}') \cup F_{ij}$, for each $i,j=1,2,3$.  The nine maximal forests $F_{ij}'$, and consequently the maximal forests $F_{ij}$ in $\F(G)$, form a Cartesian product graph $C_3 \times C_3$, which contains a cycle of length 9.  By Lemma \ref{cycle-kn}, $\F^2(G)$ contains $K_9$.
\end{proof}


We now show that repeated application of the forest graph operator to many graphs creates larger and larger complete subgraphs.

\begin{lemma}\label{l2.2}
If $G$ has a cycle of (finite) length $n$ with $n\geq4$ or it has two edge disjoint triangles, then for any finite $m\geq1$, $\F^{m}(G)$ contains $K_{m^{2}}$.
\end{lemma}

\begin{proof}
We prove this lemma by induction on $m$. 

\textbf{Case 1:}  Suppose that $G$ has a cycle $C_n$ of length $n$ ($n \geq 4$, $n$ finite).  
By Lemma \ref{cycle-kn}, $\F(G)$ contains $K_{n}$ as a subgraph, which implies that $\F(G)$ contains $K_4$.  
By Lemma \ref{k-cayley}, $\F^3(G)$ contains $K_{16}$ and in particular it contains $K_{3^2}$.

\textbf{Case 2:}  Suppose that $G$ has two edge disjoint triangles.
By Lemma \ref{triangles-k9} $\F^2(G)$ contains $K_{9}$ as a subgraph.  
It follows by Lemma \ref{l2.1} that $\F^{3}(G)$ contains $K_{\lfloor{9^2}/{4} \rfloor}=K_{20}$ as a subgraph. This implies that $\F^{3}(G)$ contains $K_{3^2}$ as a subgraph.

By Cases 1 and 2 it follows that the result is true for $m = 1, 2, 3$. 
Let us assume that the result is true for $m=l\geq3$, i.e., that $\F^{l}(G)$ contains $K_{l^{2}}$ as a subgraph.  
By Lemma \ref{l2.1} it follows that $\F(\F^{l}(G))$ has a subgraph $K_{\lfloor{l^4}/{4}\rfloor}$. Since ${\lfloor{l^4}/{4}\rfloor} > (l+1)^2$, it follows that $\F^{l+1}(G)$ contains $K_{(l+1)^2}$. By the induction hypothesis $\F^{m}(G)$ contains $K_{m^{2}}$ for any finite $m\geq1$.
\end{proof}

With Lemma \ref{k-cayley} it is clearly possible to prove a much stronger lower bound on complete subgraphs of iterated forest graphs, but Lemma \ref{l2.2} is good enough for our purposes.


\begin{lemma}\label{fg-order}
A forest graph that is not $K_1$ has no isolated vertices and no isthmi.
\end{lemma}

\begin{proof}
Let $G=\F(H)$ for some graph $H$.  Consider a vertex $F$ of $G$, that is, a maximal forest in $H$.  
Let $e$ be an edge of $F$ that belongs to a cycle $C$ in $H$.  
Then there is an edge $f$ in $C$ that is not in $F$ and $F'=F-e+f$ is a second maximal forest that is adjacent to $F$ in $G$.  Since $C$ has length at least 3, it has a third edge $g$.  If $g$ is not in $F$, let $F''=F-e+g$.  If $g$ is in $F$, let $F''=F-g+f$.  In both cases $F''$ is a maximal forest that is adjacent to $F$ and $F'$.  Thus, $F$ is not isolated and the edge $FF'$ in $G$ is not an isthmus.  

Suppose $F,F'\in\N(H)$ are adjacent in $G$.  That means there are edges $e\in E(F)$ and $e'\in E(F')$ such that $F'=F-e+e'$.  Thus, $e$ belongs to the unique cycle in $F+e'$.  As shown above, there is an $F''\in\N(H)$ that forms a cycle with $F$ and $F'$.  Therefore the edge $FF'$ of $G$ is not an isthmus.

Let $F\in\N(H)$ be an isolated vertex in $G$.  If $H$ has an edge $e$ not in $F$, then $F+e$ contains a cycle so $F$ has a neighboring vertex in $G$, as shown above.  Therefore, no such $e$ can exist; in other words, $H=F$ and $G$ is $K_1$.
\end{proof}


\section{Basic Properties of an Infinite Forest Graph}\label{properties}

We now present a crucial foundation for the proof of the main theorem in Section \ref{convergence}.  The \emph{cyclomatic number} $\beta_1(G)$ of a graph $G$ can be defined as the cardinality $|E(G)\setminus E(F)|$ where $F$ is a maximal forest of $G$.

\begin{proposition}\label{2.8}
Let $G$ be a graph such that $|\mathfrak{C}(G)| = \beta$, an infinite cardinal number.  Then:
 \begin{description}
  \item[i)] $\beta_1(G) = \beta$ and $\beta_1(\F(G))=2^\beta$.  
  \item[ii)] Both the order of $\F(G)$ and its number of edges equal $2^{\beta}$.  Both the order and the number of edges of $G$ equal $\beta$, provided that $G$ has no isolated vertices and no isthmi.
  \item[iii)] $\F(G)$ is $\beta$-regular.  
  \item[iv)] The order of any connected component of $\F(G)$ is $\beta$, and it has exactly $\beta$ edges.  
  \item[v)]  $\F(G)$ has exactly $2^{\beta}$ components.  
  \item[vi)] Every component of $\F(G)$ has exactly $\beta$ cycles.  
  \item[vii)] $|\mathfrak{C}(\F(G))| = 2^{\beta}$.  
 \end{description}
\end{proposition}

\begin{proof}
Let $G$ be a graph with $|\mathfrak{C}(G)|$ = $\beta$ ($\beta$ infinite).  

i)  Let $F$ be a maximal forest of $G$.  The number of cycles in $G$ is not more than the number of finite subsets of $E(G)\setminus E(F)$.  This number is finite if $E(G)\setminus E(F)$ is finite, but it cannot be finite because $|\mathfrak{C}(G)|$ is infinite.  
Therefore $E(G)\setminus E(F)$ is infinite and the number of its finite subsets equals $|E(G)\setminus E(F)|=\beta_1(G)$.  Thus, $\beta_1(G) \geq |\mathfrak{C}(G)|$.  
The number of cycles is at least as large as the number of edges not in $F$, because every such edge makes a different cycle with $F$.  Thus, $|\mathfrak{C}(G)| \geq \beta_1(G)$.  
It follows that $\beta_1(G) = |\mathfrak{C}(G)| = \beta$.  Note that this proves $\beta_1(G)$ does not depend on the choice of $F$.

The value of $\beta_1(\F(G))$ follows from this and part (vii).

ii) For the first part, let $F$ be a maximal forest of $G$ and let $F_0$ be a maximal forest of $G\setminus E(F)$.  
As $G\setminus E(F)$ has $\beta_1(G)=\beta$ edges by part (i), it has $\beta$ non-isolated vertices by Lemma \ref{edge-count}.  $F_0$ has the same non-isolated vertices, so it too has $\beta$ edges.

Any edge set $A \subseteq F_0$ extends to a maximal forest $F_A$ in $F\cup A$.  Since $F_A\setminus F=A$, the $F_A$'s are distinct.  Therefore, there are at least $2^\beta$ maximal forests in $F_0 \cup F$.  The maximal forest $F$ consists of a spanning tree in each component of $G$; therefore, the vertex sets of components of $F$ are the same as those of $G$, and so are those of $F_0 \cup F$.  Therefore, a maximal forest in $F_0 \cup F$, which consists of a spanning tree in each component of $F_0 \cup F$, contains a spanning tree of each component of $G$.  

We conclude that a maximal forest in $F_0 \cup F$ is a maximal forest of $G$ and hence that there are at least $2^\beta$ maximal forests in $G$, i.e., $|\N(G)| \geq 2^{\beta}$. 
Since $G$ is a subgraph of $K_{\beta}$, and since $|\N(K_{\beta})| = 2^{\beta}$ by Lemma \ref{l2.5}, we have $|\N(G)|\leq 2^{\beta}$.  
Therefore $|\N(G)|=2^{\beta}$. That is, the order of $\F(G)$ is $2^{\beta}$.  By Lemmas \ref{fg-order} and \ref{edge-count}, that is also the number of edges of $\F(G)$.

For the second part, note that $G$ has infinite order or else $\beta_1(G)$ would be finite.  If $G$ has no isolated vertices and no isthmi, then $|V(G)|=|E(G)|$ by Lemma \ref{edge-count}.  By part (i) there are $\beta$ edges of $G$ outside a maximal forest; hence $\beta\leq|E(G)|$.  

Since every edge of $G$ is in a cycle, by the axiom of choice we can choose a cycle $C(e)$ containing $e$ for each edge $e$ of $G$.  Let $\mathfrak{C}=\{C(e) : e \in E(G)\}$.  The total number of pairs $(f,C)$ such that $f\in C\in \mathfrak{C}$ is no more than $\aleph_0.|\mathfrak{C}| \leq \aleph_0.|\mathfrak{C}(G)| = \aleph_0.\beta = \beta$.  This number of pairs is not less than the number of edges, so $|E(G)|\leq\beta$.
It follows that $G$ has exactly $\beta$ edges.

iii) Let $F$ be a maximal forest of $G$.  By part (i), $|E(G)\setminus E(F)|=\beta$. 
By adding any edge $e$ from $E(G)\setminus E(F)$ to $F$ we get a cycle $C$.  Removing any edge other than $e$ from the cycle $C$ gives a new maximal forest which differs by exactly one edge with $F$.  The number of maximal forests we get in this way is $\beta_1(G)$ because there are $\beta_1(G)$ ways to choose $e$ and a finite number of edges of $C$ to choose to remove, and $\beta_1(G)$ is infinite.  Thus we get $\beta$ maximal forests of $G$, each of which differs by exactly one edge with $F$.  Every such maximal forest is generated by this construction.  Therefore, the degree of any vertex in $\F(G)$ is $\beta$.

iv) Let $A$ be a connected component of $\F(G)$.  As $\F(G)$ is $\beta$-regular by part (iii), it follows that $|V(A)|\geq \beta$.  
Fix a vertex $v$ in $A$ and define the $n^{\text{th}}$ neighborhood $D_n=\{v':d(v,v')=n\}$ for each $n$ in $\mathbb{N}$.  
Since every vertex has degree $\beta$, $|D_0|=1$, $|D_1|=\beta$ and $|D_k|\leq \beta|D_{k-1}|$. Thus, by induction on $n$, $|D_n|\leq \beta$ for $n>0$. 

Since $A$ is connected, it follows that $V(A)=\bigcup_{i\in \mathbb{N}\cup\{0\}}D_i$, i.e., $V(A)$ is the countable union of sets of order $\beta$.  Therefore $|A|=\beta$, as $|\mathbb{N}|.\beta'=\beta'$. Hence any connected component of $\F(G)$ has $\beta$ vertices.  By Lemma \ref{edge-count} it has $\beta$ edges.

v) By parts (ii, iv) the order of $\F(G)$ is $2^{\beta}$ and the order of each component of $\F(G)$ is $\beta$.  
Since $|\F(G)|=2^{\beta}$, $\F(G)$ has at most $2^{\beta}$ components.  
Suppose that $\F(G)$ has $\beta'$ components where $\beta'<2^{\beta}$.  As each component has $\beta$ vertices, it follows that $\F(G$) has order at most $\beta'.\beta=\max\{\beta',\beta \}$. 
This is a contradiction to part (ii).  Therefore $\F(G)$ has exactly $2^{\beta}$ components.

vi) Let $A$ be a component of $\F(G)$.  Since it is infinite, by part (iv) it has exactly $\beta$ edges.  
Suppose that $|\mathfrak{C}(A)|=\beta'$.  
Then $\beta'$ is at most the number of finite subsets of $E(A)$, which is $\beta$ since $|E(A)|=\beta$ is infinite; that is, $\beta'\leq\beta$.  
By the argument in part (iii) every edge of $\F(G)$ lies on a cycle.  The length of each cycle is finite.  Thus $A$ has at most $\aleph_0.\beta'=\max \{ \beta', \aleph_0 \}=\beta'$ edges if $\beta'$ is infinite and it has a finite number of edges if $\beta'$ is finite.  Since $|E(A)|=\beta$, which is infinite, $\beta'\geq\beta$.  We conclude that $\beta'=\beta$.

vii) By parts (v, vi) $\F(G)$ has $2^{\beta}$ components and each component has $\beta$ cycles.  Since every cycle is contained in a component, $|\mathfrak{C}(\F(G))| = \beta.2^{\beta}=2^{\beta}$.
\end{proof}


From the above proposition it follows that an infinite graph cannot be a forest graph unless every component has the same infinite order $\beta$ and there are $2^{\beta}$ components.  A consequence is that the infinite graph itself must have order $2^{\beta}$.  Hence, 

\begin{corollary}\label{not2power}
Any infinite graph whose order is not a power of $2$, including $\aleph_0$ and all other limit cardinals, is not a forest graph.
\end{corollary}


\begin{corollary}\label{l2.14}
For a graph $G$ the following statements are equivalent.  
 \begin{description}
  \item[i)]   $\F(G)$ is connected.
  \item[ii)]  $\F(G)$ is finite.
  \item[iii)] The union of all cycles in $G$ is a finite graph.
 \end{description}
\end{corollary}

\begin{proof}
(i)$\implies$(iii).  Suppose that $\F(G)$ is connected. If $G$ has infinitely many cycles then by Proposition \ref{2.8}(v) $\F(G)$ is disconnected. Therefore $G$ has finitely many cycles. 
Let $A = \{ e \in E(G) : \text{edge $e$ lies on a cycle in } G\}$.  Then $|A|$ is finite because the length of each cycle is finite.  That proves (iii).  

(iii)$\implies$(ii).  As every maximal forest of $G$ consists of a maximal forest of $A$ and all the edges of $G$ which are not in $A$, $G$ has at most $2^n$ maximal forests where $n=|A|$. Hence $\F(G)$ has a finite number of vertices and consequently is finite.

(ii)$\implies$(i).  By identifying vertices in different components (Whitney vertex identification; see Section \ref{roots}) we can assume $G$ is connected so $\F(G)=\T(G)$.  Cummins \cite{cummins1966hamilton} proved that the tree graph of a finite graph is Hamiltonian; therefore it is connected.
\end{proof}


\section{$\F$-Roots}\label{roots}

In this section we establish properties of $\F$-roots of graphs.  We begin with the question of what an $\F$-root should be.


Since any graph $H'$ that is isomorphic to an $\F$-root $H$ of $G$ is immediately also an $\F$-root, the number of non-isomorphic $\F$-roots is a better question than the number of labeled $\F$-roots.  
We now show in some detail that a still better question is the number of non-isomorphic $\F$-roots without isthmi.  

Let $t_\beta$ be the number of non-isomorphic rooted trees of order $\beta$.  We note that $t_{\aleph_0}\geq2^{\aleph_0}$, by a construction of Reinhard Diestel (personal communication, July 10, 2015).  
(We do not know a corresponding lower bound on $t_\beta$ for $\beta>\aleph_0$.)
Let $P$ be a one-way infinite path whose vertices are labelled by natural numbers, with root 1; choose any subset $S$ of $\mathbb{N}$ and attach two edges at every vertex in $S$, forming a rooted tree $T_S$ (rooted at 1).  Then $S$ is determined by $T_S$ because the vertices in $S$ are those of degree at least 3 in $T_S$.  (If $2\in S$ but $1 \notin S$, then vertex $1$ is determined only up to isomorphism by $T_S$, but $S$ itself is determined uniquely.)  The number of sets $S$ is $2^{\aleph_0}$, hence $t_{\aleph_0}\geq2^{\aleph_0}$.  

\begin{proposition}\label{l4.1}
Let $G$ be a graph with an $\F$-root of order $\alpha$. If $\alpha$ is finite, then $G$ has infinitely many non-isomorphic finite $\F$-roots.  
If $\alpha$ is finite or infinite, then $G$ has at least $t_\beta$ non-isomorphic $\F$-roots of order $\beta$ for every infinite $\beta \geq \alpha$.
\end{proposition}

\begin{proof}
Let $G$ be a graph which has an $\F$-root $H$, i.e., $\F(H)\cong G$, and let $\alpha$ be the order of $H$.  We may assume $H$ has no isthmi and no isolated vertices unless it is $K_1$.

Suppose $\alpha$ is finite; then let $T$ be a tree, disjoint from $H$, of any finite order $n$.  Identify any vertex $v$ of $H$ with any vertex $w$ of $T$.  The resulting graph $H_T$ also has $G$ as its forest graph since $T$ is contained in every maximal forest of $H_T$.  As the order of $H_T$ is $\alpha+n-1$ and $n$ can be any natural number, the graphs $H_T$ are an infinite number of non-isomorphic finite graphs with the same forest graph up to isomorphism.

Suppose $\alpha$ is finite or infinite and $\beta \geq \alpha$ is infinite.  Let $T$ be a rooted tree of order $\beta$ with root vertex $w$; for instance, $T$ can be a star rooted at the star center.  Attach $T$ to a vertex $v$ of $H$ by identifying $v$ with the root vertex $w$.  Denote the resulting graph by $H_T$; it is an $\F$-root of $G$ and it has order $\beta$ because it has order $\alpha+\beta$, which equals $\beta$ because $\beta$ is infinite and $\beta\geq\alpha$.  
As $H$ has no isthmi, $T$ and $w$ are determined by $H_T$; therefore, if we have a non-isomorphic rooted tree $T'$ with root $w'$ (that means there is no isomorphism of $T$ with $T'$ in which $w$ corresponds to $w'$), $H_{T'}$ is not isomorphic to $H_T$.  (The one exception is when $H=K_1$, which is easy to treat separately.)  The number of non-isomorphic $\F$-roots of $G$ of order $\beta$ is therefore at least the number of non-isomorphic rooted trees of order $\beta$, i.e., $t_\beta$.
\end{proof}

Proposition \ref{l4.1} still does not capture the essence of the number of $\F$-roots.  Whitney's \emph{2-operations} on a graph $G$ are the following \cite{whitney1933-2isom}:
\begin{enumerate}[\quad (1)]
\item \emph{Whitney vertex identification.}  Identify a vertex in one component of $G$ with a vertex in a another component of $G$, thereby reducing the number of components by 1.  For an infinite graph we modify this by allowing an infinite number of vertex identifications; specifically, let $W$ be a set of vertices with at most one from each component of $G$, and let $\{W_i: i\in I\}$ be a partition of $W$ into $|I|$ sets (where $I$ is any index set); then for each $i\in I$ we identify all the vertices in $W_i$ with each other.  
\item \emph{Whitney vertex splitting.}  The reverse of vertex identification.  
\item \emph{Whitney twist.}  If $u,v$ are two vertices that separate $G$---that is, $G=G_1\cup G_2$ where $G_1\cap G_2=\{u,v\}$ and $|V(G_1)|, |V(G_2)|>2$, then reverse the names $u$ and $v$ in $G_2$ and then take the union $G_1\cup G_2$ (so vertex $u$ in $G_1$ is identified with the former vertex $v$ in $G_2$ and $v$ with the former vertex $u$).  Call the new graph $G'$.  For an infinite graph we allow an infinite number of Whitney twists.  
\end{enumerate}
It is easy to see that the edge sets of maximal forests in $G$ and $G'$ are identical, hence $\F(G)$ and $\F(G')$ are naturally isomorphic.  It follows by Whitney vertex identification that every graph with an $\F$-root has a connected $\F$-root, and it follows from Whitney vertex splitting that every graph with an $F$-root has an $\F$-root without cut vertices.

We may conclude from Proposition \ref{l4.1} that the most interesting question about the number of $\F$-roots of a graph $G$ that has an $\F$-root is not the total number of non-isomorphic $\F$-roots (which by Proposition \ref{l4.1} cannot be assigned any cardinality); it is not the number of a given order; it is not even the number that have no isthmi; it is the number of non-2-isomorphic, connected $\F$-roots with no isthmi and (except when $G=K_1$) no isolated vertices.


We do not know which graphs have $\F$-roots, but we do know two large classes that cannot have $\F$-roots.

\begin{theorem}\label{t4.1}
No infinite connected graph has an $\F$-root.
\end{theorem}

\begin{proof}
This follows by Corollary \ref{l2.14}.
\end{proof}


\begin{theorem}\label{l4.2}
No bipartite graph $G$ has an $\F$-root.
\end{theorem}

\begin{proof} Let $G$ be a bipartite graph of order $p$ $(p\geq2)$ and let $H$ be a root of $G$, i.e., $\F(H)\cong G$.
Suppose $H$ has no cycle; then $\F(H)$ is $K_{1}$, which is a contradiction. Therefore $H$ has a cycle of length $\geq 3$. It follows by Lemma \ref{cycle-kn} that $\F(H)$ contains $K_{3}$, a contradiction. Hence no bipartite graph $G$ has a root.
\end{proof}


\section{$\F$-Convergence and $\F$-Divergence}\label{convergence}

In this section we establish the necessary and sufficient conditions for $\F$-convergence of a graph.


\begin{lemma}\label{l3}
Let $G$ be a finite graph that contains a $C_n$ (for $n\geq4$) or at least two edge disjoint triangles; then $G$ is $\F$-divergent.
\end{lemma}

\begin{proof}
Let $G$ be a finite graph.
By Lemma \ref{l2.2}, $\F^m(G)$ contains $K_{m^2}$ as a subgraph. Therefore, as $m$ increases the clique size of $\F^m(G)$ increases. Hence $G$ is $\F$-divergent. 
\end{proof}


\begin{lemma}\label{l3.4}
If $|\mathfrak{C}(G)| = \beta$ where $\beta$ is infinite, then $G$ is $\F$-divergent.
\end{lemma}

\begin{proof}
Assume $|\mathfrak{C}(G)|=\beta$ ($\beta$ infinite).  By Proposition \ref{2.8}(vii), 
as $2^{\beta}<2^{2^{\beta}}<2^{2^{2^{\beta}}}<\cdots$, it follows that $|\mathfrak{C}(\F(G))|<|\mathfrak{C}(\F^2(G))|<|\mathfrak{C}(\F^3(G))|<\cdots$. Therefore, as $n$ increases $|\mathfrak{C}(\F^n(G))|$ increases.  Hence $G$ is $\F$-divergent.
\end{proof}


\begin{theorem}\label{t3.1}
Let $G$ be a graph. Then,
\begin{description}
  \item[i)] $G$ is $\F$-convergent if and only if either $G$ is acyclic or $G$ has only one cycle, which is of length $3$.
  \item[ii)] If $G$ is $\F$-convergent, then it converges in at most two steps.
   \end{description}
\end{theorem}

\begin{proof}
i) If $G$ has no cycle, then it is a forest and $\F(G)$ is $K_1$.  If $G$ has only one cycle and that cycle has length $3$, then $\F(G)$ is $K_3$. Therefore in each case $G$ is $\F$-convergent.

Conversely, suppose that $G$ has a cycle of length greater than $3$ or has at least two triangles.  
If G has infinitely many cycles, then it follows by Lemma \ref{l3.4} that $G$ is $\F$-divergent.
Therefore we may assume that $G$ has a finite number of cycles.  If $G$ has a finite number of vertices, then it is finite and by Lemma \ref{l3} it is $\F$-divergent.  Therefore $G$ has an infinite number of vertices.  However, it can have only a finite number of edges that are not isthmi, because each cycle is finite.  Thus $G$ consists of a finite graph $G_0$ and any number of isthmi and isolated vertices.  Since $\F(G)$ depends only on the edges that are not isthmi and the vertices that are not isolated, $\F(G)=\F(G_0)$ (under the natural identification of maximal forests in $G_0$ with their extensions in $G$ by adding all isthmi of $G$).  Therefore, $G$ is $\F$-divergent.

ii) If $G$ has no cycle, then $G$ is a forest and $\F(G)\cong\F^2(G)\cong K_1$. If $G$ has only one cycle, which is of length 3, then $\F(G)\cong\F^2(G)\cong K_3$.  Therefore $G$ converges in at most 2 steps.
\end{proof}

\begin{corollary}\label{c3.1}
A graph $G$ is $\F$-stable if and only if $G = K_1$ or $K_3$.
\end{corollary}


\section{$\F$-Depth}\label{depth}

In this section we establish results about the $\F$-depth of a graph.

\begin{theorem}\label{l5.1}
Let $G$ be a finite graph.  The $\F$-depth of $G$ is infinite if and only if $G$ is $K_{1}$ or $K_{3}$.
\end{theorem}

\begin{proof}
Let $G$ be a finite graph.
Suppose that $G$ is $K_1$ or $K_3$. Then by Corollary \ref{c3.1}, it follows that $G$ is $\F$-stable. Therefore, the $\F$-depth of $G$ is infinite.

Conversely, suppose that $G$ is different from $K_{1}$ and $K_{3}$.

\textbf{Case 1:}  Let $|V|<4$.  
Then $G$ has no $\F$-root so its $\F$-depth is zero.

\textbf{Case 2:}  Let $|V|=4$.  
Suppose $G$ has an $\F$-root $H$ (i.e., $\F(H)\cong G$).  Then $H$ should have exactly $4$ maximal forests. That is possible only when $H$ has only one cycle, which is of length $4$.  By Lemma \ref{cycle-kn} it follows that $\F(H)$ contains $K_4$, hence it is $K_4$. Therefore $G$ has an $\F$-root if and only if it is $K_4$.  Hence the $\F$-depth of $G$ is zero, except that the depth of $K_4$ is 1. 

\textbf{Case 3:}  Let $|V|=n$ where $n>4$.
Suppose that $G$ has infinite $\F$-depth. Then for every $m$ there is a graph $H_{m}$ such that $\F^{m}(H_{m})=G$.  If $H_{m}$ does not have two triangles or a cycle of length greater than $3$, then $H_{m}$ has only one cycle which is of length $3$, or no cycle and $H_{m}$ converges  to $K_{1}$ or $K_{3}$ in at most two steps, a contradiction. Therefore $H_{m}$ has two triangles or a cycle of length greater than $3$.  By Lemma \ref{l2.2} it follows that $\F^{m}(H_{m})$ contains $K_{m^{2}}$ for each $m\geq2$, so that in particular $\F^{n}(H_{n})$ contains $K_{n^{2}}$.  That is, $G$ contains $K_{n^2}$.
This is impossible as $G$ has order $n$.  Hence the $\F$-depth of $G$ is finite.
\end{proof}


\begin{theorem}\label{t5.2}
The $\F$-depth of any infinite graph is finite.
\end{theorem}

\begin{proof}
Let $G$ be a graph of infinite order $\alpha$.  If $G$ has an $\F$-root, then $G$ is without isthmi or isolated vertices.

If $G$ is connected, Theorem \ref{t4.1} implies that $G$ has no root. Therefore its $\F$-depth is zero.

If $G$ is disconnected, assume it has infinite depth.  Then for each natural number $n$ there exists a graph $H_n$ such that $G\cong\F^n(H_n)$.  
Let $\beta_n$ denote the order of $H_n$.  Since $\F(H_1)\cong G$, by Proposition \ref{2.8}(ii) $\alpha=2^{\beta_1}$, from which we infer that $\beta_1<\alpha$.  
This is independent of which root $H_1$ is, so in particular we can take $H_1=\F(H_2)$ and conclude that $\beta_1=2^{\beta_2}$, hence that $\beta_2<\beta_1$.  
Continuing in like manner we get an infinite decreasing sequence of cardinal numbers starting with $\alpha$.  The cardinal numbers are well ordered \cite{kamke1950theory}, so they cannot contain such an infinite sequence.  It follows that the $\F$-depth of $G$ must be finite.
\end{proof}


\subsection*{Acknowledgements} This work was supported by DST-CMS project Lr.No.SR/S4/MS:516/07, Dt.21-04-2008.

T.~Zaslavsky also thanks the C.R.\ Rao Advanced Institute of Mathematics, Statistics and Computer Science for its hospitality during the initiation of this paper.


\end{document}